\documentclass[a4paper, 12pt,oneside,reqno]{amsart}
\usepackage[a4paper]{geometry}
\usepackage{graphicx}
\usepackage[matrix,arrow,curve,cmtip]{xy}
\usepackage{amsfonts}
\usepackage{hyperref}

\usepackage{txfonts}
\DeclareMathAlphabet{\mathcal}{OMS}{cmsy}{m}{n} % but do not change mathcal symbols

\usepackage{
  amsmath,  % Contains many useful mathematical typesetting tools.
  amssymb,  % Contains many useful mathematical symbols.
  amsthm,   % Contains tools for theorem-like environments
  tikz,     % A versatile tool to draw diagrams
  fullpage, % More reasonable margins
  youngtab, % To draw young tableaux.
  ytableau, % this is the best young tableaux packege,
  thmtools, % Lets us define theorem-like environments easily.
  hyperref, % Better references.
  cite,     % Better citations.
  url       % URL handling
}

\usepackage{amscd}
\usepackage{euscript}
\usepackage{latexsym}
\usepackage[cp1251]{inputenc}
\usepackage[english]{babel}
\usepackage{mathrsfs}
\usepackage{graphicx}
\usepackage{keyval}
\usepackage{mathtools}
\usepackage{tikz}
\usetikzlibrary{arrows,shapes,snakes,automata,backgrounds,petri,through,positioning}
\usetikzlibrary{intersections}
\usepackage[symbol*]{footmisc}
\usepackage{verbatim}
\usepackage{tikz-cd}

\usepackage{etoolbox}
\patchcmd{\thebibliography}{\section*}{\paragraph}{}{}

\numberwithin{equation}{section}

\newtheorem{theorem}{Theorem}[section]
\newtheorem{lemma}[theorem]{Lemma}
\newtheorem{proposition}[theorem]{Proposition}
\newtheorem{corollary}[theorem]{Corollary}

\theoremstyle{definition}
\newtheorem{definition}[theorem]{Definition}
\newtheorem{example}[theorem]{Example}

\title{On the pre-commutative envelopes
of commutative algebras}

\author{H. Alhussein$^{1),2),3)}$}
\author{P. Kolesnikov$^{4)}$}

\address{$^{1)}$Novosibirsk State University, Novosibirsk,  Russia.}
\address{$^{2)}$Siberian State University of Telecommunication and Informatics, Novosibirsk, Russia.}
\address{$^{3)}$Novosibirsk State University of Economics and Management, Novosibirsk,  Russia.}

\address{$^{4)}$Sobolev Institute of Mathematics, Novosibirsk, Russia.}

\begin{document}
\begin{abstract}
We prove that every 
nilpotent commutative algebra can be embedded 
into a pre-commutative (Zinbiel) algebra with respect to the anti-commutator operation. 
For finite-di\-men\-sional algebras,
the nilpotency condition is necessary 
for a commutative algebra to have a pre-commutative
envelope.
\end{abstract}

\maketitle
\section{Introduction}

The classical Poincar\'e--Birkhoff--Witt Theorem (PBW-Theorem) for Lie algebras gave rise to a series of generalizations to various 
multiplication changing functors between varieties 
of algebras. 
Namely, suppose
$V$ and $W$ are two (linear) operads governing the varieties of algebras called $V$-algebras and 
$W$-algebras, respectively. 
A morphism of operads $\omega: W \to  V$ induces
a functor from the variety of $V$-algebras to the variety of $W$-algebras: $B\mapsto B^{(\omega )}$, 
$B\in V$. 
Such functors
are called multiplication changing ones 
(see, e.g., \cite{Mikh}).
One of the most common examples is given by 
the morphism of operads 
$(-):\mathrm{Lie} \to \mathrm{As}$
such that $x_1x_2\mapsto x_1x_2-x_2x_1$. 
The corresponding functor between varieties 
$\mathrm{As}\to \mathrm {Lie}$
transforms an associative algebra $A$ into its 
commutator Lie algebra~$A^{(-)}$.

Hereinafter we do not distinguish notations for 
multilinear varieties of algebras and their 
governing operads. For an operad $V$ and a nonempty 
set $X$, denote by $V\langle X\rangle $ 
the free $V$-algebra generated by~$X$.

Every multiplication changing functor has left
adjoint functor which sends an arbitrary 
$W$-algebra $A$ to its universal enveloping $V$-algebra $U_\omega(A)$. 
If $A$ is a $W$-algebra generated by a set $X$ relative to defining relations 
$R\subset W\langle X\rangle $
then $U_\omega (A)$ is the $V$-algebra 
generated by the same set $X$ relative 
to the relations $\omega (R)\subset V\langle X\rangle $.

The canonical homomorphism
$i: A\to U_\omega(A)$  may not be
injective in general.
Moreover, $U_\omega(A)$ carries
a natural ascending filtration
relative to degrees in $i(A)$, 
and its associated graded algebra 
$\mathrm{gr}\,U_\omega(A)$ is also a $V$-algebra. 
As it was proposed in \cite{Mikh}, let us say the triple $(V, W, \omega)$ to have the PBW-property
if $i$ is injective and
\[
\mathrm{gr}\,U_\omega (A) \simeq U_\omega (A^{(0)})
\]
as $V$-algebras, where $A^{(0)}$ stands for the $W$-algebra on the space~$A$
with trivial (zero) operations.

Many combinatorial 
and homological 
properties of $V$- and $W$-algebras
are closely related if 
$(V, W, \omega)$ has the PBW-property, 
see, e.g. \cite{Mikh, DotsBMS}.

A series of operation-transforming functors 
are related with so called
dendriform splitting of varieties. 
The term ``dendriform algebra'' was introduced by 
J.-L.~Loday 
\cite{Loday2001} in the associative context, but it can 
be defined for an arbitrary variety (see, e.g., \cite{BaiGuoNi-2013, GubKol-2013}). 
Namely, for every multilinear 
variety $V$ of algebras there is a 
variety denoted 
$\mathrm{pre}\,V$. 
The defining identities of $\mathrm{pre}\,V$
can be calculated by means of a routine procedure 
(called {\em splitting})
described in \cite{BaiGuoNi-2013} 
or \cite{GubKol-2013}, 
see also \cite{GubKol-2014}. 

In particular, for every $\mathrm{pre}\,V$-algebra 
$A$ with operations $\succ$ and $\prec $ 
the same space $A$ relative to the operation 
\begin{equation}\label{eq:SplitOperation}
ab = a\succ b + a\prec b, \quad a,b\in A,
\end{equation}
is a $V$-algebra. Thus we have a morphism 
of operads $\varepsilon : V\to \mathrm{pre}\,V$
which maps $x_1x_2$ to $x_1\succ x_2 + x_1\prec x_2$.
The corresponding operation-changing functor 
between varieties of algebras is also denoted $\varepsilon $, 
so that if $A\in \mathrm{pre}\,V$ then $A^{(\varepsilon)}\in V$.

The corresponding left adjoint functor 
was previously studied for 
$V=\mathrm{Lie}$ \cite{Bokut.chen}
and $V=\mathrm{As}$ \cite{Koles}.
In both cases, the triple 
$(\mathrm{pre}\,V, V, \varepsilon)$
has the PBW-property. 

In this paper, we consider the triple 
$(\mathrm{pre}\,\mathrm{Com}, \mathrm{Com}, \varepsilon)$, 
where $\mathrm{Com}$ is the variety of associative 
and commutative algebras. In this case, 
the operad 
$\mathrm{pre}\,\mathrm{Com}$ corresponds 
to the variety of {\em Zinbiel} algebras \cite{Loday2001}, 
linear spaces with one bilinear multiplication
satisfying the identity
\begin{equation}\label{eq:ZinbIdent}
x(yz)=(xy)z+(yx)z.
\end{equation}
(This operad is Koszul dual to the operad governing 
the class of Leibniz algebras, the term ``Zinbiel'' 
is motivated by this observation.)

The functor $\varepsilon $ mentioned above 
is natural to denote $(+)$ is this particular case: 
every $\mathrm{pre}\,\mathrm{Com}$-algebra $Z$
turns into a commutative algebra 
relative to the operation 
\[
a*b = ab+ba,\quad a,b\in Z.
\]

It turns out that the case $V=\mathrm{Com} $ essentially differs from the cases $V=\mathrm{Lie}$ or $\mathrm{As}$.
It is not hard to see that not every commutative 
algebra $A$
embeds into its universal enveloping $U_{(+)}(A)$, 
so there is no hope for the PBW-property 
to hold for the triple 
$(\mathrm{pre}\,\mathrm{Com}, \mathrm{Com}, (+))$. 

It was shown in \cite{Dzhtul,2023} that a finite-dimensional pre-commutative algebra is nilpotent. 
On the other hand, it is easy to see that 
if a finite-dimensional commutative algebra $A$
embeds into a pre-commutative algebra then 
$A$ must be nilpotent. Our results, 
in particular, show the converse:
every nilpotent commutative algebra 
$A$ embeds into an appropriate 
pre-commutative algebra. 

For a trivial algebra $A$ (with zero multiplication) 
we compute the Gr\"obner--Shirshov basis of its 
universal enveloping Zinbiel algebra 
which may be considered as a ``pre-algebra analogue'' 
of the symmetric algebra of a linear space.

\section{Dendriform splitting and Zinbiel algebras}

Let $V$ be a class of all algebras satisfying a given set $\Sigma $ of multilinear identities (i.e., 
$V$ is a variety defined by $\Sigma $). 
Each identity $f\in \Sigma $ 
is an element  
from the free non-associative (magmatic) 
algebra 
$M\langle x_1,x_2,\ldots \rangle $
which is homogeneous of degree $n=\deg f$ 
and multilinear in the variables $x_1,\dots ,x_n$.
(For simplicity, we consider algebras with one binary product.)

For example, the class $\mathrm{Perm}$ is defined 
by two identities
\[
(x_1x_2)x_3 - x_1(x_2x_3), \quad x_1(x_2x_3)-x_2(x_1x_3),
\]
these are left-commutative associative algebras 
also known as Perm-algebras \cite{Chapoton}.

Construct a class of algebras $\mathrm{pre}\,V$
with two binary products as follows 
\cite{GubKol-2014}.
A linear space $A$ equipped with two operations 
denoted $\prec $ and $\succ $ belongs to 
$\mathrm{pre}\,V$ if and only if 
for every Perm-algebra $P$ the space 
$P\otimes A$
equipped with multiplication
\[
(p\otimes a)(q\otimes b) = pq\otimes (a\succ b)
+qp\otimes (a\prec b), \quad p,q\in P,\ a,b\in A,
\]
belongs to the class $V$.

In particular, if $P=\Bbbk $ then it follows 
immediately from the definition that a $\mathrm{pre}\,V$-algebra $A$ relative to the operation $(a,b)\mapsto ab$
given by \eqref{eq:SplitOperation}
is an algebra from~$V$.

The passage from a variety $V$ to $\mathrm{pre}\,V$
described above is equivalent to the procedure 
of {\em splitting} described in \cite{BaiGuoNi-2013}
in terms of Manin products for operads. 
The equivalence of these two approaches 
\cite{GubKol-2014}
was proved by means of the notion of a Rota--Baxter operator. This notion is also essential 
for our study.

\begin{definition} [see, e.g., \cite{Baxter}] \label{defn1.1}
A linear operator $R$ defined on an algebra $A$ over a field $\Bbbk$ is called a Rota-Baxter
operator (RB-operator) of  weight zero if it satisfies the relation
\[
R(x)R(y) = R(R(x)y + xR(y)),\quad  x, y \in A.
\]
An algebra A with a Rota-Baxter operator is called a Rota--Baxter algebra (RB-algebra).
\end{definition}

A Rota--Baxter operator is a formalization of 
integration. For example, 
if $A$ is an arbitrary algebra over a field 
of characteristic zero, and $A[[t]]$ 
is the algebra of formal power series over $A$
then the linear map
\[
\begin{aligned}
R: {}&  A[[t]]\to A[[t]],  \\
& \sum\limits_{n\ge 0} a_nt^n \mapsto 
  \sum\limits_{n\ge 0} \dfrac{a_n}{n+1}t^{n+1},
\end{aligned}
\]
is a Rota--Baxter operator. 
Similarly, if we restrict to the 
subalgebra $tA[[t]]$ of all series without 
free term then 
\[
R: \sum\limits_{n\ge 1} a_nt^n \mapsto 
  \sum\limits_{n\ge 1} \dfrac{a_n}{n}t^{n+1}
\]
is also a Rota--Baxter operator.

\begin{proposition}
[\cite{BaiGuoNi-2013}]\label{prop:RB-dendr}
Let $A$ be an algebra from a variety $V$
equipped with a Rota--Baxter operator 
$R: A\to A$. 
Then 
the same space $A$ with new operations 
\[
a\succ b = R(a)b, \quad a\prec b = aR(b),
\]
for $a,b\in A$, is a $\mathrm{pre}\,V$-algebra 
denoted $A_R$.
\end{proposition}

\begin{example}\label{exmp:preLie}
Let $V=\mathrm{Lie}$ be the class of Lie algebras. 
Then for every $(A,\prec, \succ)\in 
\mathrm{pre\,Lie}$ the skew-symmetry of Lie algebras 
implies $a\succ b = -b\prec a$
for all $a,b\in A$. Hence, one operation is enough 
to describe the structure of a pre-Lie algebra. 
It follows from the Jacobi identity that 
the operation $\succ $ satisfies the identity
\[
(x_1\succ x_2)\succ x_3 - (x_2\succ x_1)\succ x_3 
- x_1\succ (x_2\succ x_3) + x_2\succ (x_1\succ x_3)
\]
i.e., is left-symmetric, and the operation $\prec $
satisfies the opposite right-symmetric identity.
\end{example}

\begin{example}\label{exmp:Zinbiel}
Let $V=\mathrm{Com}$ be the class of associative 
and commutative algebras. 
Then for every $(Z,\prec, \succ)\in 
\mathrm{pre\,Com}$ the commutativity  
implies $a\succ b = b\prec a$
for all $a,b\in Z$.
Again, one operation is enough 
to describe the structure of a pre-commutative algebra. 
It follows from the associativity 
of $P\otimes Z$, $P\in \mathrm {Perm}$
that 
the operation $\succ $ satisfies the identity
\eqref{eq:ZinbIdent}:
\[
(x_1\succ x_2)\succ x_3 + (x_2\succ x_1)\succ x_3 
- x_1\succ (x_2\succ x_3),
\]
and the operation $\prec $ satisfies the opposite one.
\end{example}

\begin{definition}[\cite{Loday2001}]\label{defn:Zinbiel}
An algebra $Z$ with one binary operation 
is said to be a Zinbiel algebra 
if 
\[
a(bc) = (ab)c+(ba)c.
\]
for all $a,b,c\in Z$.
\end{definition}

Hence, a Zinbiel algebra is the same as 
a pre-commutative algebra in terms of the operation~$\succ $.
Similarly, the class of all pre-associative algebras 
coincides with the variety of dendriform algebras 
defined in \cite{Loday2001}.

As it was mentioned above, every pre-associative algebra $Z$ turns into an associative and commutative 
algebra $Z^{(+)}$ with respect to anti-commutator. 

\begin{example}\label{exmp:FreeZinbiel}
Let $X$ be a nonempty set, $X^*$ be the set of all (associative) words in the alphabet $X$ (excluding the empty word), and let $F=\Bbbk X^{*}$ be the formal linear span of $X^*$ (this is the semigroup algebra of the free semigroup generated by $X$).
Define a product on $F$ as follows:
\[
(x_1\dots x_n)(y_1\dots y_{m+1}) = 
\sum\limits_{\sigma \in S_{n,m}} \sigma (x_1\dots x_ny_1\dots y_m) y_{m+1}, 
\quad x_i,y_j\in X,
\]
where $S_{n,m}\subset S_{n+m}$ is the set of all 
$(n,m)$-shuffle permutations from the symmetric group 
$S_{n+m}$, and $\sigma (u)$, $u\in X^{n+m}$,
stands for the word obtained by corresponding permutation of letters.

Then $F$ is a pre-commutative algebra, the corresponding $F^{(+)}$ is the well-known 
{\em shuffle  algebra}
structure on the tensor algebra of the space 
$\Bbbk X$.
\end{example}

The algebra from Example \ref{exmp:FreeZinbiel}
is the free pre-commutative algebra 
generated by a set $X$ \cite{Loday2001}, its basis 
consists of right-normed monomials
\begin{equation}\label{eq:Right-normed}
( \dots ((x_1x_2)x_3) \dots x_{n})x_{n+1},
\quad x_i\in X, \ i\ge 0.
\end{equation}

The purpose of this paper was to determine if 
the nilpotence of a commutative algebra $A$
is sufficient for $A$ to be embeddable into 
an appropriate pre-commutative algebra.
As a result, we obtain a more general sufficient 
condition, but start with the trivial case when $A$
has zero multiplication. In this case, it is possible 
to compute an analogue of the PBW-basis of $U_{(+)}(A)$ by means of the Gr\"obner--Shirshov bases technique for non-associative algebras.

\section{Composition--Diamond Lemma for 
non-associative algebras}	

The Gr\"obner--Shirshov bases method for nonassociative algebras goes back 
to the paper by A. Kurosh \cite{Kurosh}, 
it is closely related with the general 
Knuth--Bendix algorithm.
An essential advance in this technique 
for Lie algebras was obtained by A. Shirshov 
\cite{shir}, for associative and commutative 
algebras the Gr\"obner bases technique 
is widely used after \cite{Buchb}.

In this section, we recall the basics of 
the Gr\"obner--Shirshov bases method for 
non-associative algebras according to \cite[Section~5]{BokChen_BMS}.

Let $\Bbbk$ be a field, $X$ be a nonempty set,
and let  $M\langle X\rangle$ stand for 
the free non-associative algebra generated by~$X$.
Suppose the set $X$ is equipped with a well order
$\leq$,
and let $X^{**}$ 
denote the set of all non-associative words in the alphabet $X$ (excluding the empty word).
The set $X^{**}$ 
is a linear basis of $M\langle X\rangle$, 
it  inherits the order  $\leq $ on $X$ in a way described below.

For any $u \in X^{**}$, denote by $|u|$ 
the length of~$u$. 
Define the {\em weight} $\mathrm{wt}\,(u)$ 
of $u\in X^{**}$ as follows: 
for $u=x\in X$ put $\mathrm{wt}\,(u)=(1,x)\in \mathbb Z_+\times X$,
for $u=(u_1 u_2)$, put 
$\mathrm{wt}\,(u)=(|u|,u_2,u_1)\in \mathbb Z_+\times X^{**}\times X^{**}$. 
Extend the initial order $\leq $ on $X$ 
to the order on $X^{**}$ by induction on the length: 
\begin{equation}\label{eq:Order-weight}
u\le v \ \Longleftrightarrow \mathrm{wt}\,(u)
\le \mathrm{wt}\,(v)
\end{equation} 
lexicographically.
That is, if $|u|<|v|$ then $u<v$, 
if $|u|=|v|=1$ then this is just the order on $X$, 
if $|u|=|v|=l>1$ then we present both $u=(u_1u_2)$, 
$v=(v_1v_2)$, where $|u_i|,|v_i|<l$, 
and then compare the factors, for which the order 
is already defined by induction.
This is a monomail order, i.e.,
\[ 
u\le v \Rightarrow wu\le wv,\ uw\le vw ,
\]
for all $u,v,w\in X^{**}$.

Every $0\ne f\in M\langle X\rangle$ may be presented as $f=\sum^n_{i=1}\alpha_i u_i$, 
where each $\alpha_i \in \Bbbk$, $\alpha_i \ne 0$, $u_i\in X^{**}$,
and $u_1>u_2>\dots>u_n$. 
The leading monomial $\bar f$ of $f\ne 0$ is then
$u_1$. 
If $\alpha_1=1$, then $f$ is said to be 
a monic polynomial.
 
\begin{definition}
Let $f,g \in M\langle X\rangle$ be monic polynomials. 
Assume there exists a word $u\in (X\cup \{\star\})^{**}$
(where $\star $ is a formal new letter not in $X$)
such that $w=\bar f$ is obtained from $u$
by replacing $\star $ with $\bar g$, 
i.e., $\bar f = u|_{\star=\bar g}$.
Then the polynomial
$(f,g)_u=f-u|_{\star =g}$ 
is called a {\em composition of inclusion}
of $f$ and $g$ with respect to~$w$.
The word $w$ as above is called an ambiguity.
\end{definition}

Let $S\subseteq M\langle X\rangle$ 
be a nonempty set of monic polynomials relative
to a monomial order $\leq $ 
on $X^{**}$. 
A polynomial $h\in M\langle X\rangle $
is said to be {\em trivial modulo $(S,w)$}, 
where $w\in X^{**}$ is a fixed word, 
if there exist a finite number of 
$u_i \in (X\cup\{\star \})^{**}$
such that 
\[
h = \sum_i  \alpha_i u_i|_{\star =s_i}, \quad s_i\in S,\ \alpha_i\in \Bbbk ,
\]
where $u_i|_{\star =\bar s_i} < w$ for all $i$.
We denote this property of $h$ as
\[
h\equiv 0 \pmod {S,w}.
\]

\begin{definition}\cite{BokChen_BMS,Bok72,shir}
A set $S$ of monic polynomials from $M\langle X\rangle $ 
is called a {\em Gr\"obner--Shirshov basis} (GSB)
if for every $f,g\in S$ we have 
$(f,g)_u \equiv 0 \pmod {(S,\bar f)}$
provided that such a composition exists.
In other words, all compositions of elements 
from $S$ are {\em trivial}.
\end{definition}

\begin{theorem}[Composition--Diamond Lemma for non-associative algebras, \cite{BokChen_BMS}]
Let $X^{**}$ be be equipped with a well monomial order $\leq $. 
For a set $S\subseteq M\langle X\rangle$  
of monic polynomials, 
the following statements are equivalent.

(i) $S$ is a Gr\"obner--Shirshov basis in $M\langle X\rangle$.

(ii) If $f\ne 0$, belongs to the ideal $I(S)$ of 
$M\langle X\rangle $ generated by $S$ then  
$\bar f =u|_{\star=\bar s}$ 
for some $s\in S$, $u\in (X\cap \{\star\})^{**}$.

(iii) The set 
\[
\mathrm{Irr}\,(S)=\{a\in X^{**}| a\ne u|_{\star=\bar s},\text{ for neither } s\in S, 
u\in (X\cap \{\star\})^{**}\}
\]
is a linear basis of the algebra 
$M\langle X\mid  S\rangle:= 
M\langle X\rangle/I(S)$.
\end{theorem}

If a subset $S$ of $M\langle X\rangle$ is not a Gr\"obner--Shirshov basis, then we can add to $S$ all nontrivial compositions of polynomials from $S$, and by continuing this process (maybe infinitely) many times, we eventually obtain
a Gr\"obner--Shirshov basis $S^{\mathrm{comp}}$.
Such a process is called the Shirshov algorithm. 

Let $X$ be a nonempty set equipped 
with a well order $\le$. 
Let us extend this order 
to a monomial order on $X^{**}$ 
as described by \eqref{eq:Order-weight}.
Then the free pre-commutative algebra $F=F(X)$ 
is defined by the following family of relations:
\begin{equation}\label{eq:Zinb_relations}
a(bc)-(ab)c-(ba)c;\ a,b,c\in X^{**}.
\end{equation}
The leading monomial is $a(bc)$ since $|bc|>|c|$.

\begin{theorem}\label{Thm:2}
The set of all polynomials \eqref{eq:Zinb_relations}
is a Gr\"obner--Shirshov basis. 
\end{theorem}

Hence, the set $\mathrm{Irr}\,(S)$
which consists of all right-normed monomials 
\eqref{eq:Right-normed}
is indeed a linear basis of the free pre-commutative algebra 
$F=M \langle X \mid S\rangle$.

\section{Universal pre-commutative envelopes
of commutative algebras}

Let $A$ be an associative and commutative algebra.
Denote by $*$ the multiplication in $A$.
If $A$ contains a non-zero idempotent $e=e*e$
then $A$ cannot be embedded into an algebra 
of the form $Z^{(+)}$, where $Z$ is a pre-commutative (Zinbiel) algebra. 
Indeed, if $\varphi: A\to Z^{(+)}$ 
is such an embedding and $x =\varphi(e)$, 
then $x = \varphi(e) = \varphi(e*e) =2xx$.
The identity \eqref{eq:ZinbIdent}
implies
\[
x(xx)=2(xx)x = xx,
\]
so $xx=2xx$ and $x=0$, a contradiction.  

Hence, the universal pre-commutative envelope 
$U=U_{(+)}(A)$ of a commutative algebra $A$
does not necessalily contains $A$ as a subalgebra of 
$U^{(+)}$. 

Let us consider the simplest case when $A$
is an algebra with trivial (zero) multiplication.
Even in this case, finding the structure of 
$U_{(+)}(A)$ requires certain computations.
We will find here an analogue 
of the  Poincar\'e--Birkhoff--Witt 
basis for the pre-commutative envelope of an
algebra $A$ such that $A*A=0$ by means 
of the  GSB method.

\begin{theorem}\label{Thm:1} 
Let $X$ be a basis of an algebra $A$ 
such that $A^2=0$, and let 
$\leq$ be a well order on $X$.
Then the following polynomials form a GSB of the universal enveloping pre-commutative 
algebra $U_{(+)}(A)$:
\begin{enumerate}
  \item[(R1)] $f_{abc} = a(bc)-(ab)c-(ba)c$, $a,b,c\in X^{**}$;
  \item[(R2)] $g_{xy}=xy+yx$, $x,y\in X$, $x<y$;
  \item[(R2$'$)] $u_x =xx$, $x\in X$;
  \item[(R3)] $t_{axy} = (ax)y+(ay)x$, $x,y\in X$, $a\in X^{**}$, $x<y$, the length of $a$ is even;
  \item[(R3$'$)] $t_{axx} = (ax)x$, $x\in X$, 
   $a\in X^{**}$, the length of $a$ is even.
\end{enumerate}
\end{theorem}
  
\begin{proof}
All compositions among the 
relations of type (R1) are trivial:
they were considered in Theorem~\ref{Thm:2}.
Hence, we may consider other relations from $S$
only with words of the form
\[
a= [z_1,z_2,\ldots,z_m]\in X^{**}, \quad z_i\in X,
\]
where $[\ldots ]$ denotes left-normed bracketing:
$[z_1,z_2,\ldots,z_m] = (((z_1z_2)z_3)\ldots)z_m$.

First, let us prove that (R3) and (R3$'$)
follow from the defining relations 
(R1), (R2), (R2$'$) of $U_{(+)}(A)$.
Suppose $a= [z_1,z_2,\ldots,z_m]$
as above, and $m$ is even. 
Proceed by induction on $m\ge 2$.
If $m=2$ then 
\begin{multline*}
ag_{xy} = (z_1z_2)(xy) + (z_1z_2)(yx) \\
=[z_1,z_2,x,y] + [z_1,x,z_2,y] + [x,z_1,z_2,y]
+ [z_1,z_2,y,x] + [z_1,y,z_2,x] + [y,z_1,z_2,x] \\
=(ax)y + (ay)x + (g_{z_1x}z_2)y + (g_{z_1y}z_2)x,
\end{multline*}
so $(ax)y + (ay)x$ follows from (R1), (R2), and (R2$'$).
Suppose
\[
ag_{xy} = [z_1,z_2,\ldots,z_m](xy) + 
[z_1,z_2,\ldots,z_m] (yx),
\]
and 
\begin{multline*}
[z_1,z_2,\ldots,z_m](xy) \equiv 
[x,z_1,z_2,\ldots,z_m,y] +
[z_1,x,z_2,\ldots,z_m,y] \\
+
\dots 
+ [z_1,z_2,\ldots,x,z_m,y]
+[z_1,z_2,\ldots,z_m,x,y]
\end{multline*}
modulo the relations (R1).
All terms except the last one 
form the pairs like
\[
[z_1,z_2,\ldots, z_{2l},x, z_{2l+1},\ldots , z_m,y]
+
[z_1,z_2,\ldots, z_{2l},z_{2l+1}, x, \ldots , z_m,y],
\quad l=1,\dots, (m-2)/2,
\]
each of them is a corollary of (R1), (R2), and (R2$'$) by induction. 
Hence, 
$a(xy)+a(yx)$ and $(ax)y+(ay)x$ both belong to 
the ideal generated by (R1), (R2), and (R2$'$).

Relations of the form (R3$'$) are proved similarly.

Now denote by $S$ the set of polynomials in the statement and prove that all their compositions 
are trivial.

The only potentially nontrivial compositions 
$(f,g)_u$ 
of inclusion are the following:
\begin{enumerate}

\item [(C1--2)]  
 $f=f_{axy}$, $g=g_{xy}$, $u = (a\star )$, 
$a\in X^{**}$ is left-normed;
\item [(C1--2$'$)]
$f=f_{axx}$, $g=u_{x}$, $u = (a\star )$, 
$a\in X^{**}$ is left-normed;
\item [(C1--3)]
$f = f_{a(bx)y}$, $g=t_{bxy}$, $u=(a\star )$, 
$a,b\in X^{**}$ are left-normed; 
\item [(C1--3$'$)]
$f = f_{a(bx)x}$, $g=t_{bxx}$, $u=(a\star )$, 
$a,b\in X^{**}$ are left-normed.

\end{enumerate}

Consider the composition (C1--2):
\begin{multline*}
(f,g)_u = f_{axy}-ag_{xy} 
=a(xy)-(ax)y-(xa)y-a(xy)-a(yx) \\
  = -(ax)y-(xa)y-(ay)x-(ya)x .\\
\end{multline*}
Here $a= [z_1,z_2,\ldots,z_m]$.

Assume $m$ is an even number. Then 
\begin{multline*}
f_{axy}-ag_{xy} = -(xa)y-(ya)x-t_{axy}
    = -(x[z_1,z_2,\ldots,z_m])y-(y[z_1,z_2,\ldots,z_m])x
     -t_{axy}\\
\equiv
-
\sum\limits_{\sigma\in S_{1,m-1}}
  \big (\sigma ([x,z_1,\ldots , z_{m-1},z_m])y
+
\sigma ([y,z_1,\ldots , z_{m-1},z_m])x \big )
\\
= -[g_{x,z_1},z_{2},\ldots,z_m,y]
    -\sum_{l=1}^{\frac{m-2}{2}}
    [z_1,\ldots, z_{2l},x,z_{2l+1},\ldots,z_m,y]-[z_{1},\ldots,z_{2l},z_{2l+1},x,z_{2l+2},\ldots,z_m,y]
    \\
 -[g_{yz_1},z_{2},\ldots,z_m,y]
 -\sum_{l=1}^{\frac{m-2}{2}}
    [z_1,\ldots, z_{2l},y,z_{2l+1},\ldots,z_m,x]-[z_1,\ldots, z_{2l}, z_{2l+1},y,z_{2l+2},\ldots,z_m,x] \\
\equiv 
 -\sum_{l=1}^{\frac{m-2}{2}}
 [((a_{2l}x)z_{2l+1}+(a_{2l}z_{2l+1})x),z_{2l+2},\ldots,z_m,y]
  -\sum_{l=1}^{\frac{m-2}{2}}
  [((a_{2l}y)z_{2l+1}+(a_{2l}z_{2l+1})y),z_{2l+2},\ldots,z_m,x] \\
 = 
 -\sum_{l=1}^{\frac{m-2}{2}}
 [t_{a_{2l},x,z_{2l+1}},z_{2l+2},\ldots,z_m,y]
  -\sum_{l=1}^{\frac{m-2}{2}}
  [t_{a_{2l},y,z_{2l+1}},z_{2l+2},\ldots,z_m,x]
\equiv 0 \pmod {S, a(xy)}.
\end{multline*} 
Here $a_{2l} = [z_1,\ldots , z_{2l}]$ is a left-normed word of even length.

Now assume $m$ is an odd number. Then, similarly,
rewrite the composition into left-normed form 
to obtain 
\begin{multline*}
f_{axy}-ag_{xy}
 = -[z_1,z_2,\ldots,z_m,x,y]
   -[z_1,z_2,\ldots,z_m,y,x] 
   -(x[z_1,z_2,\ldots,z_m])y
   -(y[z_1z_2,\ldots,z_m])x
\\
\equiv
    -[z_1,z_2,\ldots,z_m,x,y]
    -[x,z_1,\ldots , z_m,y]
    -\sum_{l=1}^{m-1}[a_{l},x,z_{l+1},\ldots,z_m,y] \\
    -[z_1,z_2,\ldots,z_m,y,x]
    -[y,z_1,\ldots, z_m,x]
    -\sum_{l=1}^{m-1}[a_{l},y,z_{l+1},\ldots,z_m,x]
\\
\equiv
    -\sum_{l=2}^{m}[a_{l},x,z_{l+1},\ldots,z_m,y] 
    -\sum_{l=2}^{m}[a_{l},y,z_{l+1},\ldots,z_m,x]
\\
    = -\sum_{i=1}^{\frac{m-1}{2}}[a_{2i},x,z_{2i+1},\ldots,z_m,y]-[a_{2i},z_{2i+1},x,z_{2i+2},\ldots,z_m,y]
    \\
     -\sum_{i=1}^{\frac{m-1}{2}}[a_{2i},y,z_{2i+1},\ldots,z_m,x]-[a_{2i},z_{2i+1},y,z_{2l+2},\ldots,z_m,x]
\\
     =-\sum_{i=1}^{\frac{m-1}{2}} 
     \big (
     [t_{a_{2i}xz_{2i+1}},z_{2i+2},\ldots, z_m,y]+[t_{a_{2i}yz_{2i+1}},z_{2i+2},\ldots, z_m,x] \big )
   \equiv 0 \pmod {S, a(xy)}.
\end{multline*}
Here, as above, 
$a_l = [z_1,\ldots, z_l]$.

Therefore, all compositions of type (C1--2) are trivial. For (C1--2$'$), the same computations 
show triviality of such compositions.

To complete the proof, we need the following 

\begin{lemma}\label{lem:Odd-Even-Zero}
Let $a=[x_1,x_2,\ldots,x_m]$, 
$b=[y_1,y_2,\ldots,y_k]\in X^{**}$, $m\ge 1$ is odd and 
$k\ge 2$ is even.
Then there exist $u_i\in X^{**}$, 
$s_i\in S$, $\alpha_i\in \Bbbk $ such that
\[
ab = \sum_i  \alpha_i u_i|_{\star =s_i}, 
\]
where $u_i|_{\star =\bar s_i} \le ab$ for all $i$.
\end{lemma}

We will say $ab$ is trivial if such a presentation exists.

\begin{proof}
For $m=1$, $k=2$ we have 
\[
ab = x_1(y_1y_2) = f_{x_1y_1y_2} + g_{x_1y_1} y_2 .
\]
Assume $k>2$ and the statement is true for $m=1$ and for all words shorter than~$k$. Then present $b = [b_{k-2},y_{k-1},y_k]$
and write
\begin{multline*}
ab = x_1((b_{k-2}y_{k-1})y_k) \equiv 
 [x_1,b_{k-2},y_{k-1},y_k] + [b_{k-2},x_1,y_{k-1},y_k]+
[b_{k-2},y_{k-1},x_1,y_k] \\
=
[(x_1b_{k-2}),y_{k-1},y_k] +((b_{k-2}x_1)y_{k-1} - (b_{k-2}y_{k-1})x_1 ) y_k.
\end{multline*}
Hereinafter $\equiv $ means the reduction by means of 
the relations $f_{abc}$.
The first summand in the right-hand side 
is trivial by induction, the second one is equal to 
$t_{b_{k-2}x_1y_{k-1}} y_k$, so it is also trivial.

Next, assume $m>1$ and $k=2$. Then present 
$a = [a_{m-2},x_{m-1},x_m]$, $a_{m-2}$ is of odd length, and write 
\begin{multline*}
ab = [a_{m-2},x_{m-1},x_m](y_1y_2)\\
\equiv 
[(a_{m-2}x_{m-1}), x_m,y_1,y_2] 
+  [(a_{m-2}x_{m-1}), y_1, x_m, y_2] 
+  [y_1, (a_{m-2}x_{m-1}), x_m, y_2] \\
= t_{(a_{m-2}x_{m-1}),x_m,y_1} + [(y_1(a_{m-2}x_{m-1})),x_m,y_2]
\end{multline*}
The second summand is trivial by induction (the case $m=1$), hence, 
the entire expression is trivial.

Finally, assume $m>1$, $k>2$, and the lemma is true for 
all words $a$, $b$ such that either $a$ shorter than $m$ or for  $b$ shorter than $k$. 
Then present 
$a = a_{m-1}x_m$, where $a_{m-1}$ is of even length, 
$b = [y_1,\ldots, y_{k}]$,
calculate $(a_{m-1}x_m)[y_1,\ldots, y_k]$,
and re-arrange the summands to get
\begin{multline}\label{eq:SumX3}
ab = (a_{m-1} x_m)[y_1,\ldots, y_{k}]
\equiv 
\sum\limits_{j=1}^k \sum\limits_{i=j}^k
 [y_1,\ldots , y_{j-1}, a_{m-1}, y_j, \ldots, y_{i-1}, x_m, y_{i}, \ldots, y_{k}] 
 \\
 =
\sum\limits_{p=0}^{k/2-1}
\sum\limits_{i=2p+1}^{k}
 [y_1,\ldots, y_{2p}, a_{m-1},y_{2p+1}, \ldots, y_{i-1},x_m, 
 y_{i},\ldots, y_k] \\
 +
\sum\limits_{p=0}^{k/2-1}
\sum\limits_{i=2p+2}^{k}
 [y_1,\ldots, y_{2p}, y_{2p+1}, a_{m-1},y_{2p+2}, \ldots, y_{i-1},x_m, 
 y_{i},\ldots, y_k] .
\end{multline}
The first group of summands in the right-hand side of \eqref{eq:SumX3}
may be presented as
\begin{multline*}
\sum\limits_{p=0}^{k/2-1}
[u_{p,2p},x_m,y_{2p+1}, \ldots, y_k] 
+ [u_{p,2p},y_{2p+1},x_m,y_{2p+2}, \ldots, y_k]  \\
+
[u_{p,2p+2},x_m,y_{2p+3}, \ldots, y_k] 
+ [u_{p,2p+2},y_{2p+3},x_m,y_{2p+4}, \ldots, y_k] + \dots \\
+ \dots 
+
[u_{p,k-2}, x_m, y_{k-1}, y_k] 
+
[u_{p,k-2}, y_{k-1}, x_m y_k] \\
=
[t_{[u_{p,2p}x_my_{2p+1}}, y_{2p+2}, \ldots, y_k]
+
[t_{[u_{p,2p+2}x_my_{2p+3}}, y_{2p+4}, \ldots, y_k]
+ \dots + 
t_{[u_{p,k-2} x_m y_{k-1}}y_k \equiv 0 \pmod {S,ab},
\end{multline*}
where 
$u_{p,i} = [y_1,\ldots, y_{2p}, a_{m-1},y_{2p+1}, \ldots, y_{i}]$.
All summands in the second group contain factors of the form
\[
[y_1,\ldots, y_{2p+1}]a_{m-1}
\]
which are trivial by induction (the length of $a_{m-1}$ is even).
\end{proof}

Proceed to the compositions of type (C1--3). 
Consider $f_{a(bx)y}, t_{bxy} \in S$, 
the length of $b$ is even.
Then for $w=(a\star)$ we have 
\begin{multline*}
(f_{a(bx)y}, t_{bxy})_w  = f_{a(bx)y}-at_{bxy}
 = a((bx)y)-(a(bx))y-((bx)a)y-a((bx)y)-a((by)x) \\
     = -((ab)x)y -((ba)x)y-((bx)a)y
      -((ab)y)x -((ba)y)x-((by)a)x
\end{multline*}
Suppose $m$ is even.
Then both
$ab$, $ba$ are linear combinations of words which have even length. If $ab+ba = \sum\limits_{j\ge 0} \alpha_ju_j$, 
$\alpha_j\in \Bbbk $, then 
\[
 ((ab)x)y +((ba)x)y + ((ab)y)x +((ba)y)x \equiv 
 \sum\limits_{j\ge 0} \alpha_jt_{u_jxy} \equiv 0 
 \pmod {S, a((bx)y)}.
\]
The remaining terms in the composition are
$((by)a)x + ((bx)a)y$.
They both contain factors $(bx)a$ or $(by)a$
that are trivial by Lemma~\ref{lem:Odd-Even-Zero}.

Suppose $m$ is an odd number. Then, modulo Lemma~\ref{lem:Odd-Even-Zero},
the remaining terms of the composition are
\[
h = [b,a,x,y] + [b,x,a,y] + [b,a,y,x] + [b,y,a,x] .
\]
Let us rewrite $h$ as follows:
\begin{multline*}
  [b,a,x,y] + [b,x,a,y] + [x,b,a,y] +
  [b,a,y,x] + [b,y,a,x] + [y,b,a,x]
  - [x,b,a,y] - [y,b,a,x] \\
\equiv (ba)(xy) + (ba)(yx) -[(xb,a,y]-[(yb),a,x]
= (ba)u_{xy}  -[(xb,a,y]-[(yb),a,x] .
\end{multline*}
All summands are trivial by Lemma~\ref{lem:Odd-Even-Zero}, 
and all monomials here are smaller than $a((bx)y)$
since $b$ is a non-empty word.

In a similar way, the composition (C1--3$'$) is also trivial.
\end{proof}

\begin{corollary}
If $A$ is a linear space with an ordered basis $X$
then the set
\[
\{[x_1,x_2,x_3,x_4,\ldots, x_{n-1},x_n ] \mid
x_i\in X, \, n \ge 1,\, 
x_1>x_2,   x_3>x_4, \dots \}
\]
(with no restrictions on $x_n$ if $n$ is odd)
is a linear basis of the algebra $U_{(+)}(A)$
if $A$ is considered as an algebra with zero multiplication.  
\end{corollary}

In other words, the universal pre-commutative envelope of a trivial algebra $A$ is isomorphic 
as a linear space to 
\[
T(A\wedge A)\otimes (A\oplus A\wedge A), 
\]
where $T(A\wedge A)$ is the tensor algebra of 
the space $A\wedge A$.

The following example shows that structure of $U_{(+)}(A)$
essentially depends on the multiplication in $A$
even if $A^n=0$ for some $n\geq 3$.  

\begin{example} 
Suppose $A$ is a commutative nilpotent algebra
with a multiplication $*$, and let 
$X=\{x_1,x_2,\ldots,x_n\}$ be a basis of $A$
such that $x_i*x_j = x_{i+j}$ or zero, if $i+j>n$.
Namely, 
$A\simeq t\Bbbk [t]/(t^{n+1})$, 
$x_i = t^i+(t^{n+1})$.
Then the defining relations of 
$U_{(+)}(A)$ are
\begin{enumerate}
  \item $a(bc)=(ab)c+(ba)c$, $a,b,c\in X^{**}$;
  \item $x_ix_j+x_jx_i=x_{i+j}$, $i+j\leq n$, $x_i,x_j\in X$;
  \item $x_ix_j+x_jx_i=0$, $i+j> n$, $x_i,x_j\in X$.
\end{enumerate}
In order to get a GSB, we have to add the following nonassociative polynomials:
\[
  x_ix_j=\frac{j}{i+j}x_{i+j},\ i+j\leq n,\ x_i,x_j\in X;
\]
\[
  x_ix_j=0,\ i+j> n, \ x_i,x_j\in X.
\]
\end{example}

Indeed, since 
$x_i(x_1x_1)=\frac{1}{2}x_ix_2$ 
and  $x_i(x_1x_1)=(x_ix_1+x_1x_i)x_1=x_{i+1}x_1$, 
we have $x_ix_2=2x_{i+1}x_1$.
By induction on $j\ge 2$, assume that $x_ix_j=jx_{i+j-1}x_1$ for all~$i$, then
\[ 
x_ix_{j+1}=
x_i(x_jx_1+x_1x_j)=x_{i+j}x_1+x_{i+1}x_j
=x_{i+j}x_1+jx_{i+j}x_1=(j+1)x_{i+j}x_1.
\]
Next, 
\[
x_{i+j}=x_ix_j+x_jx_i=jx_{i+j-1}x_1+ix_{i+j-1}x_1=(i+j)x_{i+j-1}x_1.
\]
Therefore,
$x_ix_j=jx_{i+j-1}x_1=\frac{j}{i+j}x_{i+j}$ 
for all $i$, $j$ such that 
$i+j\leq n$.

Finally, suppose $i+j> n$ 
and $j+1\leq n$, then  
\[
x_i(x_jx_1)=\frac{1}{j+1}x_ix_{j+1},\quad
x_{i}(x_jx_1)=(x_ix_j+x_jx_i)x_1=0
\]
Hence, $x_ix_{j+1}=x_{j+1}x_i=0$. Similarly, 
it can be obtained that 
$x_ix_n=x_nx_i=0$ for $i+1<n$. 

Therefore, in this particular case we have
$U_{(+)}(A)^{(+)}\simeq A$
in contrast to the case when $A$ has zero multiplication.

\section{Embedding of nilpotent algebras into Zinbiel algebras}

The main purpose of this section is to prove that 
a nilpotent commutative algebra embeds into its 
universal enveloping Zinbiel algebra although 
there is no PBW-property. 

Let us say that an algebra $A$ has a positive filtration 
if there is a descending chain of subspaces 
\[
A=F^1A \supset F^2A \supset \dots \supset F^nA \supset F^{n+1}A \supset \dots 
\]
such that 
$F^iA\cdot F^jA \subseteq F^{i+j}A$
and $\bigcap\limits_{n\ge 1} F^nA = 0$. 

For example, if $A$ is a nilpotent algebra then such a filtration exists: $F^iA = A^i$, $i=1,2,\dots $.

For every algebra with a positive filtration 
one may construct its associated graded algebra 
in the ordinary way:
\[
\mathrm{gr}\,A = \bigoplus _{n\ge 1} F^{n}A/F^{n+1}A, 
\quad (a+F^{i+1}A)(b+F^{j+1}A) = ab + F^{i+j+1}A, 
\]
for $a\in F^iA$, $b\in F^jA$.
The linear space $A$ is naturally isomorphic to the space $\mathrm{gr}\,A$. If the isomorphism preserves multiplication then we say the filtered algebra $A$
is graded.

\begin{theorem}\label{Thm4.2}
For every commutative algebra 
$A$ 
with a positive filtration
there exists a Zinbiel algebra $B$ such that
$A$ is a subalgebra of $B^{(+)}$.
\end{theorem}

\begin{proof}
First, choose a basis $X$ of the space $A$ agreed with 
the filtration,
 i.e.,
\[
X=X_1\cup X_2 \cup \ldots ,
\]
where $\bigcup_{i\geq k}X_i$ is a basis of $F^kA$ for $k\ge 1$.
Denote by $*$ the multiplication in $A$. 
If $x\in X_k$ and $y\in X_m$ then 
$x*y$ belongs to the linear span 
of $X_{k+m}\cup X_{k+m+1}\cup \ldots$, so there 
is a unique (finite) presentation 
\[
x*y = (x*y)_{k+m} + (x*y)_{k+m+1} + \dots , 
\]
where $(x*y)_i$ is in the linear span of $X_i$.

Next, consider the set 
\[
\hat{X}=\{ x_{i}^{(k)} \mid 
 x\in X_k, k\ge 1, i\ge k \}
\]
and construct the 
polynomial algebra $\Bbbk [\hat X]$. 
This algebra is graded: the degree function 
of a monomial is given by 
\begin{equation}\label{eq:degreeWt}
\deg x^{(k_1)}_{i_1}x^{(k_2)}_{i_2}\ldots x^{(k_m)}_{i_m}
= i_1+\dots + i_m.
\end{equation}
Consider the set $\hat{S}$ of the following elements in $\Bbbk [\hat{X} ]$: 
\[
s_l(x,y) =  \sum_{i+j=l} x^{(k)}_i y^{(m)}_{j}-\sum_{p=m+k}^{l}(x*y)^{(p)}_{l}
\]
where $x\in X_k$, $y\in X_{m}$, $k,m \ge 1$, $l\ge k+m$.

Since all polynomials in 
$\hat{S}$ are homogeneous relative to the degree 
function \eqref{eq:degreeWt}, 
the algebra 
$\hat{A}=\Bbbk [ \hat{X} ]/(\hat{S})$ 
inherits the grading:
\[
\hat A = \bigoplus _{n\ge 1} \hat A_n, 
\]
where $\hat A_n$ is spanned 
by the images of all monomials $u$ such that $\deg u=n$.

It was shown in \cite{Bergman} that, in noncommutative setting, the set $\hat S$ is a Gr\"obner basis 
in $\Bbbk [\hat X]$ relative to a certain ordering 
of monomials. In the commutative case, the same 
statement remains valid. In particular, 
every linear linear form (a nontrivial linear 
combination of elements from $X$) is nonzero in $\hat A$.

Finally, consider the algebra of formal power series 
(without constant terms) 
$t\hat A[[t]]$ equipped with the following Rota--Baxter 
operator:
\begin{equation}\label{eq:RB_series}
R: \sum\limits_{n\ge 1} f_n t^n \mapsto \sum\limits_{n\ge 1} \dfrac{1}{n}f_n t^n, \quad f_n\in \hat A.
\end{equation}
Define the linear mapping 
\[
    \varphi : A \to t\hat{A}[[t]]
\]
as follows: for $x\in X_k$, let
\[
 \varphi(x)=\sum_{i\ge k} i x^{(k)}_i t^i .
\]
This is an injective map since the set $\hat X$
is linearly independent in $\hat A$.

The commutative algebra $t\hat A[[t]]$ equipped with 
the Rota--Baxter operator \eqref{eq:RB_series}
is a Zinbiel algebra: $B=t\hat A[[t]]_R$. 
It is straightforward to check that 
$\varphi $ is a homomorphism of algebras.
Indeed, let $x\in X_k$, $y\in X_m$, then
\[
R(\varphi(x))\varphi(y)
= \sum_{i\geq k} x^{(k)}_it^i 
 \sum_{j\geq m}  j  y^{(m)}_j t^{j}
 = 
\sum_{l\geq k+m}  \sum_{i+j=l} 
j x^{(k)}_i y^{(m)}_j t^{l},
\]
and, similarly,
\[
\varphi(x)R(\varphi(y))  
= \sum_{l\geq k+m}  \sum_{i+j=l} 
i x^{(k)}_i y^{(m)}_j t^{l},
\]
so 
\[
R(\varphi(x))\varphi(y) + \varphi(x)R(\varphi(y))
=
\sum_{l\geq k+m}  \sum_{i+j=l} 
 l x^{(k)}_i y^{(m)}_j t^{l}
=
\sum_{l\geq k+m} l \bigg( \sum_{p= k+m}^l (x*y)_l^{(p)}
\bigg ) t^l.
\]
On the other hand,
\[
\varphi(x*y)=
\sum_{p\geq k+m} \varphi( (x*y)^{(p)} )=
\sum_{p\geq k+m}\sum_{l\geq p}l(x*y)^{(p)}_{l} t^l
=\sum_{l\geq k+m}\sum_{p=k+m}^l l(x*y)^{(p)}_{l} t^l.
\]
Hence, 
$R(\varphi(x))\varphi(y) + \varphi(x)R(\varphi(y)) = \varphi(x*y)$ as required.
\end{proof}

\begin{corollary}
Every commutative algebra  
with positive filtration embeds into 
its universal enveloping Zinbeil algebra.
\end{corollary}

\begin{proof}
Suppose $i: A\to U_{(+)}(A)$  is the canonical 
homomorphism from $A$ to its universal 
enveloping Zinbiel algebra. Then for every 
Zinbiel algebra $B$ and for every homomorphism 
$\varphi : A\to B^{(+)}$ 
there exists a unique homomorphism 
$\psi : U_{(+)}(A) \to B$ of Zinbiel algebras 
such that $\psi(i(a))=\varphi(a)$ for every $a\in A$. 
If $i$ was not injective then so is $\varphi $, 
but at least one injective $\varphi $
exists by Theorem~\ref{Thm4.2}.
\end{proof}
      
\begin{corollary}
Suppose $A$ is a finite-dimensional commutative algebra. Then algebra $A$ is embedded into a Zinbiel algebra if only if $A$ is nilpotent.
\end{corollary}

\begin{proof}
If $A$ is nilpotent then use Theorem~\ref{Thm4.2} 
applied to the standard positive filtration.   
Conversely, if $A$ is not nilpotent 
If $A$ is not nilpotent then it contains a 
non-zero idempotent, e.g., lifted from the identity element 
of $A/\mathrm{rad}\,(A)$. The presence 
of an idempotent prevents an embedding of $A$
into a pre-commutative algebra.
\end{proof}

 \subsection*{Acknowledgments}
 The study was supported by a grant from the Russian Science Foundation No.~23-71-10005, \url{https://rscf.ru/project/23-71-10005/}. The authors are grateful 
to V.~Yu. Gubarev for discussions and useful comments. The second author acknowledges the hospitality of UAEU where an essential part of the work 
was done.


\begin{thebibliography}{99}

\bibitem{Mikh}
A. A. Mikhalev, I. P. Shestakov, 
PBW-pairs of varieties of linear algebras, Comm. Algebra, 42(2), 667--687 (2014).

\bibitem{DotsBMS}
V. Dotsenko, P. Tamaroff, 
Endofunctors and Poincare--Birkhoff--Witt theorems, 
Int. Math. Res. Not. IMRN, no.~16, 12670--12690 (2021).

\bibitem{Loday2001}
J.-L. Loday, 
Dialgebras, 
In: Loday J.-L. , Frabetti A., Chapoton F., Goichot F. (Eds), Dialgebras and related operads, 
Springer-Verl., Berlin, 7--66 (2001).
(Lectures Notes in Math., vol. 1763).

\bibitem{BaiGuoNi-2013}
C. Bai, O. Bellier, L. Guo, X. Ni, 
Splitting of operations, Manin products, 
and Rota-Baxter operators, 
Int. Math. Res. Notices, no.~3, 485--524 (2013).

\bibitem{GubKol-2013}
V. Gubarev, P. Kolesnikov, 
Embedding of dendriform algebras into Rota--Baxter
algebras, 
Cent. Eur. J. Math., 11(2), 226--245 (2013).

\bibitem{GubKol-2014}
V. Gubarev, P. Kolesnikov, 
Operads of decorated trees and their duals, 
Comm. Math. Universitatis Carolinae, 55(4), 421--445
(2014).

\bibitem{Bokut.chen}
L. A. Bokut, Y. Chen, Y. Li,
Gr\"obner--Shirshov bases for 
Vinberg--Koszul--Gerstenhaber right symmetric algebras, 
J. Math. Sci. (N. Y.), 166(5), 603--612 (2010).

\bibitem{Koles}
P. Kolesnikov, 
Gr\"obner--Shirshov bases for pre-associative algebras, 
Comm. Algebra, 45(12), 5283--5296 (2017).

\bibitem{Dzhtul}
A. Dzhumadildaev, K. Tulenbaev, 
Nilpotency of Zinbiel algebras, 
Journal of
Dynamical and Control Systems, 
11(2), 195--213  (2005).

\bibitem{2023}
D. Towers, 
Zinbiel algebras are nilpotent,
J. Algebra Appl. 22(8), Paper No.~2350166 (2023).

\bibitem{Chapoton}
F.~Chapoton,
Un endofoncteur de la cat\'egorie des op\'erades,
In: Loday J.-L. , Frabetti A., Chapoton F., Goichot F. (Eds), Dialgebras and related operads, 
Springer-Verl., Berlin, 105--110 (2001).
(Lectures Notes in Math., vol. 1763).

\bibitem{Baxter}
G. Baxter, 
An analytic problem whose solution follows from a simple algebraic identity, 
Pacific J. Math., 10, 731--742 (1960).

\bibitem{Kurosh}
A. G. Kurosh, 
Nonassociative free algebras and free products of algebras,
Mat. Sb., New. Ser. 20(62), 239--262 (1947).

\bibitem{shir}
A. I. Shirshov, Some algorithmic problem for Lie algebras, Sibirsk. Mat. Z., 3 (1962) 292-296 (in
Russian). English translation: SIGSAM Bull., 33 (2)(1999), 3-6.

\bibitem{Buchb}
B. Buchberger, 
An algorithmical criteria for the solvability of algebraic systems of equations,
Aequat. Math., 4, 374--383 (1970).

\bibitem{BokChen_BMS}
L. A. Bokut, Y. Chen,
Gr\"obner--Shirshov bases and their calculation
Bull. Math. Sci., 4, 325--395 (2014).
%DOI 10.1007/s13373-014-0054-6

\bibitem{Bok72}
L. A. Bokut, 
Imbeddings into simple associative algebras, 
Algebra Logika, 15, 117--142 (1976).

\bibitem{Bergman}
G.~M. Bergman, D.~J. Britten, F.~W. Lemire, 
Embedding rings in completed graded rings 3. 
Algebras over general k, 
Journal of Algebra, 84(1),  42--61 (1983).

\end{thebibliography}
\end{document}